\documentclass[11pt]{amsart}

\usepackage{amsfonts,amssymb,amsmath,mathrsfs}
\usepackage{amsthm}
\usepackage{graphicx,epic,eepic}
\usepackage{epsfig}
\usepackage{latexsym}
\usepackage{color}

\newtheorem{theorem}{Theorem}[section]
\newtheorem{lemma}[theorem]{Lemma}
\newtheorem{proposition}[theorem]{Proposition}
\newtheorem{corollary}[theorem]{Corollary}
\theoremstyle{definition}
\newtheorem{definition}[theorem]{Definition}
\newtheorem{example}[theorem]{Example}

\theoremstyle{remark}
\newtheorem{remark}[theorem]{Remark}

\numberwithin{equation}{section}

\def\neweq#1{\begin{equation}\label{#1}}
\def\endeq{\end{equation}}
\def\eq#1{(\ref{#1})}

\newcommand{\R}{{\mathbb R}}

\newcommand{\axfrac}{\frac{({\bf a}x,x)}{|x|^2}}
\newcommand{\agxfrac}{\frac{({\bf a}_g(x)x,x)}{|x|^2}}
\newcommand{\eps}{\varepsilon}

\newcommand{\g}{\gamma}

\newcommand{\esssup}{\mathop{\mathrm{ess\,sup}}}
\newcommand{\essinf}{\mathop{\mathrm{ess\,inf}}}
\newcommand{\expn}[1]{{\exp\left\{{#1}\right\}}}
\newcommand{\Env}{{\mathop{Env}}}
\newcommand{\env}{{\mathop{env}}}
\def\Eenv{\mathcal{M}}
\def\eenv{m}

\newcommand{\abs}[1]{\lvert#1\rvert}

\date{\today}

\begin{document}

\title[Semi-linear second--order non-divergence type elliptic inequalities]
{Singular solutions for second--order non-divergence type
elliptic inequalities in punctured balls}

\author{Marius Ghergu}
\address{School of Mathematical Sciences, University College Dublin, Belfield, Dublin 4, Ireland}
\email{marius.ghergu@ucd.ie}

\author{Vitali Liskevich}
\address{Department of Mathematics,
Swansea University, Swansea SA2 8PP, United Kingdom}
\email{V.A.Liskevich@swansea.ac.uk}

\author{Zeev Sobol}
\address{Department of Mathematics,
Swansea University, Swansea SA2 8PP, United Kingdom}
\email{z.sobol@swansea.ac.uk}

\subjclass{Primary 35J60, 35B33; Secondary 35B05}



\keywords{Non-divergence semi-linear equation, singular solutions,
punctured ball, critical exponent}

\begin{abstract}
We study the existence and nonexistence of positive singular solutions to
 second--order non-divergence type elliptic inequalities in the form
\[
\sum_{i,j=1}^N a_{ij}(x)\frac{\partial  ^2 u}{\partial x_i
\partial x_j}+\sum_{i=1}^N b_i(x)\frac{\partial  u}{\partial x_i} \ge K(x) u^p,
\]
$-\infty<p<\infty$, with measurable coefficients in a punctured ball $
B_R\setminus\{0\}$ of $\R^N$, $N\geq 1$. We prove the existence of
a critical value $p^*$ that separates the existence region from
non-existence. In the critical case $p=p^*$ we show that the
existence of a singular solution depends on the rate at which the
coefficients $(a_{ij})$ and $(b_i)$ stabilize at zero and we
provide some optimal conditions in this setting.

\end{abstract}

\maketitle

%
\section{Introduction and the main results}

In this paper we are concerned  with the existence and nonexistence
of positive singular solutions to semi-linear second--order non-divergence
type elliptic inequality 
$$
{ \mathscr L}u\ge K(x) u^p \quad\mbox{almost everywhere in }\:
B_R\setminus\{0\},\eqno{(1)_p}
$$
where $B_R$ is the open ball of radius $R>0$ in $\R^N$ ($N\ge 1$)
centered at the origin, $-\infty<p<\infty$,
and ${\mathscr L}$ is given by
\neweq{operator}
{\mathscr L}u=\sum_{i,j=1}^{N} a_{ij}(x)\frac{\partial^2
u}{\partial x_i\partial x_j}+\sum_{i=1}^N b_i(x)\frac{\partial
u}{\partial x_i} .
\endeq
The matrix ${\bf a}=(a_{ij}(x))_{i,j=1}^{N}\in L^\infty(B_R)$ is
a.e. symmetric and uniformly elliptic, in the sense that there
exists a constant $\nu>1$ and such that for
almost all $x\in B_R\setminus\{0\}$
\begin{equation}\label{elliptic}
\nu^{-1}|\xi|^2 \le\sum_{i,j=1}^{N}a_{ij}(x)\xi_i\xi_j\le \nu |\xi|^2
\mbox{ for all }\xi\in\R^N.
\end{equation}
The vector ${\bf b}=(b_i(x))_{i=1}^N\in L^\infty_{\mathrm{loc}}(B_R)$ is assumed to satisfy
\begin{equation}\label{bvec}
|b_i(x)|\le \frac{c}{|x|}\quad\mbox{  for almost all }x\in
B_R\setminus\{0\}\,, 1\le i\le N,
\end{equation}
with some $c>0$. Finally, the weight $K\in L^\infty_{loc}(B_R\setminus\{0\})$ satisfies $\essinf K>0$.

In this paper we are concerned with singular solutions of $(1)_p$ in the following sense.

\begin{definition}\label{defsol}
We say that $u>0$ is a solution to $(1)_p$ if there exists $R>0$
such that $u\in W^{2,N}(B_R\setminus \overline B_{\eps})$ for all $\eps>0$ and $u$ satisfies
$(1)_p$ a.e. on $B_R\setminus \{0\}$. 

A solution $u$ to $(1)_p$ is
called a {\it singular solution} if it has a singularity at the
origin in the sense that $\limsup\limits_{|x|\rightarrow 0}u(x)=\infty$.
\end{definition}

We start with the following observation, which can be readily
verified.
\begin{proposition}
  Let $u$ be a singular solution to $(1)_p$ for some $p>1\, (p<1)$. Then $v:=u^{\frac{p-1}{q-1}}$
  is a singular solution of $(1)_q$ for $1< q <p\; (p<q<1)$.
  \end{proposition}
 The above proposition allows us to  define two critical exponents
\begin{equation}
\label{crit-exp}
\begin{split}
p_* & :=\sup\{p<1:~(1)_p \mbox{ has no singular solution}\}, \\
p^* & :=\inf\{p>1:~(1)_p \mbox{ has no  singular solution}\}.
\end{split}
\end{equation}

Then
$-\infty\leq p_*\leq 1 \leq p^*\leq +\infty$ and $(1)_p$ has a  singular solution
for $p\in (p_*,p^*)$ and $(1)_p$ has no  singular solutions for $p\in (-\infty, p_*)\cup(p^*,+\infty)$.

The aim of this research is to obtain estimates on the critical exponents introduced in~\eqref{crit-exp} and to establish the existence/non-existence of a positive singular solution to $(1)_p$ for the critical $p$. We also provide some interesting examples where the critical exponents are computed explicitly.

The case ${\mathscr L}=\Delta$, namely singular solutions (sub-solutions) to the equation
\neweq{laplace}
\Delta u=u^p\quad\mbox{ in }B_R\setminus\{0\}
\endeq
has been extensively studied during recent decades (see, e.g., \cite{bp, bv, cmv, KLS, KLSU, serrin1, serrin2, serrin3, v1} and the references therein).
By now it is well known that $p_*=-\infty$ and $p^*=\frac{N}{N-2}$ for $N\ge 3$, and $p^*=+\infty$ for $N=1,2$.

On the other hand, elliptic non-divergence type equations are the subject in its own right with numerous applications in many parts of mathematics. The study of singular solutions stems from the seminal work of Gilbarg and Serrin \cite{GS} where the behavior of solutions to ${\mathscr L}u\geq 0$ around a singular point is investigated. For semi-linear problems one of the important issues is the stability of critical exponents under small perturbations of the coefficients, and the existence of a singular solution in the critical case. In Corollary~\ref{claplace} below we show that, if the coefficients of ${\mathscr L}$ stabilize at zero to that of the Laplacian, then the critical exponents remain unchanged, but the existence of a singular solution in the critical case will depend on the speed of convergence of the coefficients. Naturally, another issue is studying operators significantly different from the Laplacian, in search of new phenomena.
For instance, by Theorem \ref{generalpositive} and Theorem \ref{critic} below, the problem
\neweq{pert}
\Delta u+\beta\frac{x}{|x|^2}\cdot \nabla u\geq u^p\quad\mbox{ in }B_R\setminus \{0\}\,,
\endeq
has singular solutions for all $-\infty<p<\infty$ if $\beta\leq 2-N$ while if $\beta>2-N$ then singular solutions exist if and only if $-\infty<p<(N+\beta)/(N+\beta-2)$.

Another motivation to our study is that, with a change of the independent variable, the problem~\eqref{laplace} as well as inequality \eq{pert} take the form $(1)_p$.

In order to formulate the results we need some additional
notation. For a measurable function $f:B_R\to \mathbb{R}$ we
define its upper and lower radial envelopes as follows:
\begin{equation}\label{envelopes}
      \Env{f}(r):= \lim\limits_{\delta\to0}\esssup\limits_{\abs{|x|-r}<\delta}f(x),\quad
      \env{f}(r):= \lim\limits_{\delta\to0}\essinf\limits_{\abs{|x|-r}<\delta}f(x).
\end{equation}

Set
\begin{equation}\label{ellip-f}
\Psi(x):=\frac{\mathop{Tr}{\mathbf a}(x)+{\mathbf b}(x)\cdot { x}}{\displaystyle \axfrac}\;,\qquad
\Theta(x):= \frac{K(x)}{\displaystyle \axfrac}.
\end{equation}

The quantity $\Psi(x)$ was introduced in \cite{MeSer} in the context of second order
non-divergent elliptic operators in exterior domains. Since $\Psi\equiv N$ for ${\bf a}\equiv\mathrm{I}$ and ${\bf b}\equiv0$, the quantity $\Psi$ is called the {\it effective dimension}. In this work we show that, similar to the problems studied in \cite{KLS1,MeSer}, the asymptotic of $\Psi$ at the origin, revealed via the envelopes $\Env{\Psi}$ and $\env{\Psi}$, plays the same role in estimating the critical exponent for $(1)_p$ for a general operator ${\mathscr L}$, as the dimension $N$ does in case ${\mathscr L}=\Delta.$

\begin{remark}\label{Psi}
%
%

For a radially symmetric function $u=u(|x|)$ one has
\begin{equation}\label{rad1}
{\mathscr L} u= \axfrac
\left(u^{\prime\prime}(|x|)+\frac{\Psi(x)-1}{|x|}u^\prime(|x|)\right).
\end{equation}
Hence a singular solution $v$ to the following inequality
\begin{equation}\label{rad2}
v^{\prime\prime}-\tfrac{\Env\Psi(r)-1}{r}\big(v^\prime\big)_-
+ \tfrac{\env\Psi(r)-1}{r}\big(v^\prime\big)_+
\geq \Env\Theta(r)v^p \quad\mbox{ in } (0,R),
\end{equation}
gives rise to a radially symmetric solution $u(x)=v(|x|)$ to $(1)_p$.
Vice versa, a solution $v$ to the following inequality
\begin{equation}\label{rad3}
v^{\prime\prime}+\tfrac{\Env\Psi(r)-1}{r}\big(v^\prime\big)_+
- \tfrac{\env\Psi(r)-1}{r}\big(v^\prime\big)_-
\leq \env\Theta(r)v^p \quad\mbox{ in } (0,R),
\end{equation}
gives rise to a radially symmetric function $u(x)=v(|x|)$
satisfying ${\mathscr L} u\leq K(x) u^p$ in $B_R\setminus\{0\}$.

\end{remark}

It is easy to see that the asymptotic behavior at the origin
of $\Theta$ and $\Psi$ is invariant under orthogonal
transformations of $x$ but in general it is not invariant under
affine transformations of $x$. However, under a transformation
$g\in GL_N$, the operator ${\mathscr L}$ is transformed into a
second order operator ${\mathscr L}_g$ with $\mathbf a_g$ and
$\mathbf b_g$ replacing $\mathbf a$ and $\mathbf b$, respectively,
\[
{\bf a}_g(x):=g\,{\bf a}(g^{-1}x)\,g^\top\,,\quad {\bf b}_g(x):=g{\bf b}(g^{-1}x).
\]
We define $\Theta_g$ and $\Psi_g$ and their envelopes in the same way as in \eqref{ellip-f} and \eqref{envelopes}.
In particular,
\begin{equation}\label{PsiG}
\displaystyle
\Psi_g(x):=\frac{\mathop{Tr}{\mathbf a}_g(x)+{\mathbf b}_g(x)\cdot { x}}{\displaystyle \agxfrac}\;,\qquad
\Theta_g(x):= \frac{K(x)}{\displaystyle \agxfrac}.
\end{equation}
In the following, we shall suppress $g$ if $g=\mathrm{I}$.
For any $g\in GL_N$ we introduce the upper and lower dimensions:
\begin{equation}\label{dimensions}
\begin{aligned}
\mathcal{N}(g)&:=\limsup\limits_{r\to 0}\tfrac1{\abs{\ln r}}\int_r^R\Env\Psi_{g}(r)\frac{dr}r\,,\quad
&&\overline{\Psi}:=\inf\limits_{g\in GL_N} \mathcal{N}(g),\\
{n}(g)&:=\liminf\limits_{r\to 0}\tfrac1{\abs{\ln r}}\int_r^R\env\Psi_{g}(r)\frac{dr}r\,,\quad
&&\underline{\Psi}:=\sup\limits_{g\in GL_N} n(g).
\end{aligned}
\end{equation}


We start first with the simpler case $K\in L^\infty (B_R)$. In this setting we have:

\begin{theorem}\label{generalpositive}
Assume $K\in L^\infty (B_R)$. Then $p_*=-\infty$ and
\neweq{ppqqw0}
1+\frac2{(\overline{\Psi}-2)_+}\leq p^*\leq
1+\frac2{(\underline{\Psi}-2)_+},
\endeq
with the convention $1/0=+\infty$.
\end{theorem}

\begin{remark}\label{simpler_Psi}
Note the following estimates:
\begin{equation}\label{simpler_Psi_eq}
\begin{split}
\overline{\Psi}\leq \overline{\overline{\Psi}}:= \inf\limits_{g\in GL_N}\lim\limits_{r\to0}\esssup\limits_{x\in B_r}\Psi_g(x),\\
\underline{\Psi}\geq \underline{\underline{\Psi}}:= \sup\limits_{g\in GL_N}\lim\limits_{r\to0}\essinf\limits_{x\in B_r}\Psi_g(x).
\end{split}
\end{equation}
So in~\eqref{ppqqw0} one can replace $\overline{\Psi}$ and
$\underline{\Psi}$ with $\overline{\overline{\Psi}}$ and $\underline{\underline{\Psi}}$, respectively, thus obtaining a weaker estimate with more transparent bounds.
\end{remark}

As a simple consequence of Theorem~\ref{generalpositive}, we
obtain that, for the coefficients stabilizing at zero, the
critical exponents coincide with those for the Laplacian.
\begin{corollary}\label{claplace}
Let $N\geq3$ and $K\in L^\infty (B_R)$. Assume that there exists a matrix $\mathbf{a}_0$ such that
\[
\lim\limits_{r\to0}\esssup_{x\in B_{r}}\Big(\big|\mathbf{a}(x) -  \mathbf{a}_0\big| + |x|\big|\mathbf{b}(x)\big| \Big)=0.
\]
Then $p_*=-\infty$ and $p^*=N/(N-2)$.
\end{corollary}
Theorem~\ref{generalpositive} leaves unclarified the case of $\Psi$ oscillating
around 2. The following result shows that in this framework the critical exponent
$p^*$ is highly unstable.
\begin{theorem}\label{unstable}
Let $K\in L^\infty(B_R)$ and $1< q\leq\infty$. Then, there exist $\mathbf{a}$ and $\mathbf{b}$
satisfying \eq{elliptic} and \eq{bvec}, respectively,  such that
\[
p^*=q
\]
and $\Psi_g$ oscillates around 2 for all $g\in GL_N$, in the sense that
\[
\underline{\underline{\Psi}}<2\leq \overline{\overline{\Psi}}.
\]

\end{theorem}

The case where $\Psi$ is not oscillating around 2
allows us to consider singular potentials $K(x)$ which behaves essentially like $|x|^{-\sigma}$ around the origin.  
Our next result extends Theorem \ref{generalpositive} in the following way:

\begin{theorem}\label{general}  Let $p_*$ and $p^*$ be defined by~\eqref{crit-exp}. Assume there exists $\sigma\geq 0$ such that 
\neweq{essentially}
\liminf_{|x|\rightarrow 0} |x|^{\sigma-\varepsilon}K(x)>0\,\mbox{ and }\; \limsup_{|x|\rightarrow 0} |x|^{\sigma+\varepsilon}K(x)<\infty\,,\mbox{ for all }\varepsilon>0.
\endeq

Then $p_*=-\infty$ and
\begin{equation}\label{ppqqw}
\begin{cases}
1+\frac{2-\sigma}{(\overline{\Psi}-2)_+}\leq p^*\leq
1+\frac{2-\sigma}{(\underline{\Psi}-2)_+},& \mbox{for } 0\leq \sigma< 2\mbox{ and } \lim\limits_{\rho\to
0}\essinf\limits_{r<\rho}\Env\Psi(r) >2;\\
p^*=1,& \mbox{for } \sigma>2 \mbox{ or }
\sigma=2,\underline\Psi>2,
\end{cases}
\end{equation}
with the convention $1/0=+\infty$ and $\overline{\Psi}$ and $\underline{\Psi}$ defined by
\eqref{dimensions}.
\end{theorem}
It is easily seen that the potentials
\[
K(x)=|x|^{-\sigma}\;, \quad K(x)=|x|^{-\sigma}\ln^{s}\frac{1}{|x|}\;,\quad K(x)=|x|^{-\sigma}\Big(2+\sin\frac{1}{|x|}\Big)\ln^{s}\Big(\ln\frac{1}{|x|}\Big),
\] 
 satisfy \eq{essentially} for $\sigma\geq 0, s\in \R$. In general, if $K(x)$ satisfies \eq{essentially}  then, so does the function $\Theta(x)$ defined in \eq{ellip-f}.

Next we are concerned with the existence of a singular solution to
$(1)_p$ in the critical case $p=p^*$. The following result shows that
in this framework the existence is related to the rate at which $\Psi$ stabilizes as
$x\to 0$.

\begin{theorem}\label{critic}
Assume $\underline\Psi=\overline{\Psi}=A\geq2$ and either $K(x)$ satisfies \eq{essentially} for some $0\leq \sigma\leq 2$ or $K(x)\in L^\infty(B_R)$ (case in which we shall take $\sigma=0$ in the following). Suppose  
there exist $h, H \in L_{\mathrm{loc}}^\infty(0,R]$ such that
\begin{equation}\label{ggw}
r^{-A}h(r)\leq \eenv(r)\leq \Eenv(r)\leq r^{-A}H(r)\quad\mbox{for
a.a. }r\in(0,R).
\end{equation}

(i) If $A>2>\sigma$ and $h$ satisfies
\begin{equation}\label{optimal}
\int_0^R h^{\frac{2-\sigma}{A-2}}(r)\frac{dr}{r}=\infty,
\end{equation}
then $(1)_p$ has no singular solution in $B_R\setminus\{0\}$ for
the critical value $p=\frac{A-\sigma}{A-2}$.

(ii) If $A>2>\sigma$,  $\lim\limits_{\rho\to
0}\essinf\limits_{r<\rho}\Env\Psi(r) >2$ and  $H$
satisfies
\begin{equation}\label{pessimal}
\int_0^R H^{\frac{2-\sigma}{A-2}}(r)\frac{dr}{r}<\infty,
\end{equation}
then $(1)_p$ has singular solutions in $B_R\setminus\{0\}$ for the
critical value $p=\frac{A-\sigma}{A-2}$.

(iii) If $A=\sigma=2$ and $h\in L^\infty(0,R)$ satisfies 
\begin{equation}\label{22}
\int_0^R h^\varepsilon(r)\frac{dr}{r}=\infty,
\end{equation}
for some $\varepsilon>0$,  then $(1)_p$ has no singular solutions in $B_R\setminus\{0\}$ for
all $p>1$, that is $p^*=1$.
\end{theorem}

\begin{remark}\label{rem}

(i) Comparison of \eqref{optimal} with \eqref{pessimal}
demonstrates the sharpness of the result. In particular, the assertion
\eq{optimal} (resp. \eq{pessimal}) holds if there exist
$c>0$, $R'\in (0,R)$ and $0<\varkappa\leq
\frac{A-2}{2-\sigma}$ (resp. $\varkappa>\frac{A-2}{2-\sigma}$)
such that
\neweq{parth}
h(r)\geq \frac c{|\ln r|^\varkappa}\;\quad\Big(\mbox{ resp. \;}
H(r)\leq \frac c{|\ln r|^\varkappa} \Big) \quad \mbox{for almost all }r\in(0,R').
\endeq
Also \eq{22} holds if there exist $c,\delta>0$ and
$R'\in (0,R)$ such that
\[
\lambda(r)\geq \frac c{|\ln r|^\delta} \quad \mbox{for almost all }r\in(0,R').
\]

(ii) Theorem \ref{general} and Theorem \ref{critic} (iii) cover
all possible situations if $\essinf Env\Psi>2$,
$\underline\Psi=\overline\Psi=A$ and $\sigma\geq 0$. Indeed,  in this case we have
\[
p^*=\begin{cases}
\infty,& \mbox{for } A\leq 2, \sigma<2;\\
1,&\mbox{for} A\leq 2,\sigma\geq 2;\\ 
1+\frac{(2-\sigma)_+}{A-2},& \mbox{for }  A>2. 
\end{cases}
\]



\end{remark}

Set
\begin{equation}\label{MandN}
\begin{split}
\Eenv_g(r)&:=\expn{\int_r^{R}\Env\Psi_g(\tau)\frac{d\tau}\tau},\\
\eenv_g(r)&:=\expn{\int_r^{R}\env\Psi_g(\tau)\frac{d\tau}\tau}.
\end{split}
\end{equation}
Both Theorem \ref{general} and Theorem \ref{critic} are consequences of the following general result.

\begin{theorem}\label{main_technical}
Let $p>1$. For $g\in GL_N$, let $\Theta_g$, $\Psi_g$ be as in~\eqref{PsiG}
and $\eenv_g$ and $\Eenv_g$ be as in~\eqref{MandN}.

(i)  There exists a singular solution to $(1)_p$ provided
\begin{equation}\label{exist_criterion}
    \int_0^{R}\left(\int_r^{R_0}\Eenv_g(\rho)\rho\,d\rho\right)^{p}\frac{\Env\Theta_g(r)}{r\Eenv_g(r)}dr<\infty.
\end{equation}

(ii) There is no singular solutions to $(1)_p$ provided
\begin{equation}\label{non-exist_criterion}
    \int_0^{R}\eenv_g^{p-1}(r)\env\Theta_g(r)r^{2p-1}dr=\infty.
\end{equation}

\end{theorem}
The proof of the existence part in Theorem \ref{main_technical}
relies on constructing a radial singular solution to $(1)_p$, more
precisely, a solution $v$ to~\eqref{rad2} on the interval $(0,R)$
such that $v(r)\to\infty$ as $r\to0$. The non-existence part in Theorem
\ref{main_technical} is achieved by comparison with a radial
barrier which is a solution to~\eqref{rad3} (see
Proposition~\ref{nonex-0} for the argument). Both constructions
lead us to the study of the final value problem for the
Emden-Fowler-type equation~\eqref{aaw4}. A major step in our
approach is to show that solutions to~\eqref{aaw4} can be extended
to the interval $(0,R)$. To this aim, we adapt classical ideas
(see \cite{bellman53}, \cite{kig}, \cite{mambriani29},
\cite{talia}) to our singular setting. An interesting feature of
our construction of the barrier which cannot be extended up to
zero, is that we use PDEs techniques such as the Harnack
inequality for the Fuchsian type operators (see
\cite{pinchover94}) and the Keller-Osserman-type
estimate~\eqref{Keller}.

\begin{remark}\label{criteria_comparison}

(i) Assumptions \eqref{exist_criterion} and \eqref{non-exist_criterion} are mutually exclusive. Indeed, we should note that   
for $r\leq \frac12R$ we have
\[
\int_r^{R}\Eenv(\rho)\rho\,d\rho
>
\int_r^{2r}\Eenv(\rho)\rho\,d\rho
> 2^{-\|\Env\Psi\|_\infty-1}r^2\Eenv(r).
\]
So
\[\begin{split}
\int_0^{\tfrac12R}\left(\int_r^{R}\Eenv(\rho)\rho\,d\rho\right)^{p}\Env\Theta(r)\frac{dr}{r\Eenv(r)}
>
&
c_p\int_0^{\tfrac12R}
    \Eenv^{p-1}(r)
    \Env\Theta(r)r^{2p-1}dr
\\
\geq
&
\int_0^{\tfrac12R}
    \eenv^{p-1}(r)\env\Theta(r)
    r^{2p-1}dr.
\end{split}
\]
Thus, \eqref{exist_criterion} and \eqref{non-exist_criterion} are mutually exclusive.

(ii) Clearly, $\Eenv_g\sim \eenv_g$ if and only if
\begin{equation}\label{equiv}
    \int_0^{R}\big(\Env\Psi_g(r) - \env\Psi_g(r)\big)\frac{dr}r<\infty.
\end{equation}
However, this does not make \eqref{exist_criterion} and \eqref{non-exist_criterion} an alternative
since in general $\int_r^{R}\Eenv_g(\rho)\rho\,d\rho \not\sim r^2\Eenv_g(r)$.
On the other hand, $\int_r^{R}\Eenv_g(\rho)\rho\,d\rho \sim r^2\Eenv_g(r)$ provided
$\essinf\Env\Psi_g >2$. Indeed, we have
\[
\begin{split}
\int_r^{R} & \Eenv_g(\rho)\rho\,d\rho =
R^2\int_r^{R}\mathrm{exp}\Big\{\int_\rho^{R}\big(\Env\Psi_g(\tau)-2\big)\frac{d\tau}\tau\Big\}\frac{d\rho}\rho
\\
\leq &
\tfrac{R^2}{\essinf\Env\Psi_g - 2}
\int_r^{R}\frac{\Env\Psi_g(\rho)-2}\rho
\expn{\int_\rho^{R}\big(\Env\Psi_g(\tau)-2\big)\frac{d\tau}\tau}
d\rho
\\
= &
\tfrac{1}{\essinf\Env\Psi_g - 2} \big( r^2\Eenv_g(r) - R^2\big).
\end{split}
\]
Obviously, it suffices to have 
$
\lim\limits_{\rho\to
0}\essinf\limits_{r<\rho}\Env\Psi_g(r) >2$,
since we do not impose restrictions on the choice of $R$.

(iii) In applications it is often complicated to verify the conditions \eq{exist_criterion} and \eq{non-exist_criterion} in Theorem~\ref{main_technical}
because of oscillating $\Psi$. It turns out that in~\eqref{exist_criterion} and~\eqref{non-exist_criterion} the envelopes of $\Psi$ can be replaced by their averages (see Appendix~A).

\end{remark}
%
%
%
%

Using Remark \ref{criteria_comparison} (ii) we deduce

\begin{corollary}\label{coro1}
Assume $\lim\limits_{\rho\to
0}\essinf\limits_{r<\rho}\Env\Psi(r) >2$ and $p>1$. There exists a singular solution to $(1)_p$ provided
\begin{equation}\label{non-exist_criterion00}
    \int_0^{R}\Eenv_g^{p-1}(r)\Env\Theta_g(r)r^{2p-1}dr<\infty.
\end{equation}

\end{corollary}

The rest of the paper is organized as follows. Section 2 contains some preliminary results concerning $(1)_p$. Sections 3 is devoted to the study of an Emden-Fowler equation which we use in the proof of Theorem \ref{main_technical}. Section 4 contains the proofs of all our main results while in Section 5 we give several examples that illustrate our findings in Theorem~\ref{general} and Theorem~\ref{critic}. In Section 6 we present some open problems that arise from our approach to $(1)_p$.

\section{Some auxiliary results}
In this section we collect some preliminary results regarding inequality
$(1)_p$. Our first result in this sense shows that a singular solution $u$ of $(1)_p$ is in some sense non-increasing around zero.
\begin{proposition}\label{l2}
Let $u$ be a singular solution of $(1)_p$ in $B_R\setminus\{0\}$ and for all $0<r<R$ denote $M(r)=\max_{|x|=r}u(x)$. Then, there exists a sequence $\{R_k\}$ of positive real numbers converging to zero such that
\[
M(R_k)<M(r)\quad\mbox{ for all }r\in(0,R_k).
\]
\end{proposition}
\begin{proof}
Assume the contrary. Then there exists $R'\in (0,R)$ such that
for every $r\in(0,R')$ there exists $r'\in (0,r)$ that satisfies $M(r')\leq M(r)$. Then by the maximum principle $u(x)\leq M(r)$ in $B_r\setminus B_{r'}$ and hence $M(s)\leq M(r)$ for all $r'\leq s\leq r$.
Thus $r\longmapsto M(r)$ is nondecreasing on $(0,R')$ which contradicts the fact that $\limsup_{|x|\rightarrow 0}u(x)=\infty$.
\end{proof}

The finite value problem for equation~\eqref{rad3} is the main tool for
our proof of non-existence of a singular solution to $(1)_p$, as shown in the
next result.

\begin{proposition}\label{nonex-0}
Let $p>1$. Assume there exists $R_0>0$ such that, for all $R\in(0,R_0)$ and $M>0$ there exist
$\lambda>0$ such that the unique local solution $v$ to the final value problem for the equation~
\eqref{rad3} with $v(R)=M$, $v'(R)=\lambda$ does not continue to $r=0$ (that is, there exists
$R'\in(0,R)$ such that  $v(r)\to\infty$ as $r\searrow R'$).
Then $(1)_p$ has no singular solutions.
\end{proposition}

\begin{proof}
Assume that there exist $R_1>0$ and a singular solution $u$ to $(1)_p$ in $B_{R_1}\setminus\{0\}$. Then, by Proposition \ref{l2}, there exists $R\in(0,R_0\wedge R_1)$ such that
\[
\max\limits_{|x|=r}u(x)>M:=\max\limits_{|x|=R}u(x) \mbox{ for all }r\in(0,R).
\]
Let $v$ be as above. Then the domain
\[
\Omega:=\{x \in B_R\setminus \overline B_{R'} :v(|x|)<u(x)\}
\]
 is non-empty  and $u(x)\leq v(|x|)$ for $x\in\partial \Omega$.
However, it follows from Remark \ref{Psi} that, with $\tilde u(x):=v(|x|)$, one has
\[
\mathscr L (u-\tilde u) \geq K(x)( u^p - \tilde u^p) \geq 0 \mbox{ in }\Omega.
\]
So $u(x)\leq  v(|x|)$ for $x\in\Omega$ by the maximal principle, which contradicts the definition of $\Omega$.
\end{proof}

The following estimate of a well-known type will be required for construction
of a solution to~\eqref{rad3}, as described in Proposition~\ref{nonex-0}.

\begin{proposition}\label{KO}{\sf (Keller-Osserman type estimate \cite{keller,osserman})}
Let $p>1$ and $K\in L^\infty(B_R)$. Then, there exists $C=C(N,R, K, p)>0$ such that any
solution $u$ of $(1)_p$ satisfies
\begin{equation}\label{Keller}
u(x)\leq C|x|^{\frac{2}{1-p}}\quad\mbox{ in
}B_{2R/3}\setminus\{0\}.
\end{equation}
\end{proposition}
\begin{proof} We use an idea from \cite{KLS1} that goes back to
\cite{KLa}. Without loss of generality we may assume $K\equiv 1$. For $0<r<R$ let us set
\[
y=\frac{x}{r} \quad\mbox{ and }\quad v(y)=r^{\frac{2}{p-1}}u(x).
\]
Then $v$ satisfies
\[
\widetilde{\mathscr L}v:=\sum_{i,j=1}^{N} \widetilde a_{ij}(y)\frac{\partial^2
v}{\partial y_i\partial y_j}+\sum_{i=1}^N \widetilde b_i(y)\frac{\partial
v}{\partial y_i}
\ge v^p\quad\mbox{ in
}B_1\setminus\{0\}\,,
\]
where
\[
\widetilde a_{ij}(y)=a_{ij}(ry)\quad\mbox{ and } \quad \widetilde b_i(y)=rb_i(ry).
\]
Note that $\widetilde{\bf a}=(\widetilde a_{ij})$ and $\widetilde {\bf b}=(\widetilde b_i)$ satisfy similar properties to \eq{elliptic} and \eq{bvec}.
Now let
\[
w(y):=c\left[\Big(\frac{9}{16}-|y|^2\Big)\Big(|y|^2-\frac{1}{16}
\Big) \right]^{\frac{2}{1-p}}\,,
\]
where $c>0$ is taken such that $\widetilde{\mathscr L}w\le w^p$ in
$B_{3/4}\setminus B_{1/4}$. Since
\[
w=\infty\quad\mbox{ on }\partial (B_{3/4}\setminus B_{1/4})\,,
\]
it follows that
\[
v\le w\quad\mbox{ in }B_{3/4}\setminus B_{1/4}.
\]
In particular
\[
v(y)\leq \max_{1/3<|y|<2/3} w(y)\quad \mbox{ in }B_{3/4}\setminus
B_{1/4}\,,
\]
which yields
\[
\max_{r/3<|x|<2r/3}u(x)\leq Cr^{\frac{2}{1-p}}.
\]
Since $0<r<R$ was arbitrarily chosen this implies \eq{Keller}.
\end{proof}

%

Our last result in this section concerns a particular type of matrix ${\bf a}$ that satisfies \eq{elliptic} and will be used later in the proof of Theorem \ref{unstable} as well as in the construction of some examples in Section \ref{examples}.

\begin{lemma}\label{computations}
Let $R>0$ and let $\beta,\gamma \in L^\infty(0,R)$,
\[\lim\limits_{r\rightarrow 0}\essinf\limits_{\tau\in(0,r)}\gamma(\tau)>-1.\] Denote
\[
\overline\gamma:= \lim\limits_{r\rightarrow 0}\esssup\limits_{\tau\in(0,r)}\gamma(\tau),~
\underline\gamma:=\lim\limits_{r\rightarrow 0}\essinf\limits_{\tau\in(0,r)}\gamma(\tau),\quad
\overline\beta:= \lim\limits_{r\rightarrow 0}\esssup\limits_{\tau\in(0,r)}\beta(\tau),~
\underline\beta:=\lim\limits_{r\rightarrow 0}\essinf\limits_{\tau\in(0,r)}\beta(\tau).
\]
We assume that, for every couple of limit points $(\overline\gamma, \overline\beta)$,
$(\underline\gamma, \overline\beta)$, $(\overline\gamma, \underline\beta)$ and $(\underline\gamma,\underline\beta )$, there exists a common sequence $r_n\to0$ realizing
both limits.

Set
\begin{equation}\label{matgs}
\mathbf{a}(x) :=\mathrm{I} + \gamma(|x|) \frac{x\otimes x}{|x|^2},
\qquad \mathbf{b}(x) :=\beta(|x|) \frac{x}{|x|^2},
\end{equation}
that is,
\[
\mathbf{a}_{ij}(x) = \delta_{ij}+ \gamma(|x|) \frac{x_ix_j}{|x|^2},
\quad \mathbf{b}_{k}(x) = \beta(|x|) \frac{x_k}{|x|^2},
\quad i,j,k=1,2,\dots,N.
\]
Then, for $\Psi$, $\underline{\underline\Psi}$ and $\overline{\overline\Psi}$ defined as in \eq{ellip-f} and \eqref{simpler_Psi_eq}, one has
\neweq{doipsi}
\begin{split}
\Psi(x) & = 1 + \frac{N-1 + \beta(|x|)}{1+\gamma(|x|) },\\
\overline{\overline{\Psi}} & =1 + (N-1+ \overline{\beta})
\begin{cases}
(1+\underline\gamma)^{-1}& \;\mbox{ if }\overline{\beta}\geq 1- N,\\
(1+\overline\gamma)^{-1}& \;\mbox{ if } \overline{\beta}< 1- N;
\end{cases}
\\
\underline{\underline{\Psi}} & = 1+ (N-1+ \underline{\beta})
\begin{cases}
(1+\overline\gamma)^{-1}& \;\mbox{ if } \underline{\beta}\geq 1- N,\\
(1+\underline\gamma)^{-1}& \;\mbox{ if } \underline{\beta}< 1- N.
\end{cases}
\end{split}
\endeq
\end{lemma}
\begin{remark}\label{remga}
Note that, for a bounded measurable function $\gamma:(0,R)\to \mathbb{R}$,
\[
\liminf_{r\to 0}\gamma(r)=\lim_{r\to 0}\essinf\limits_{\tau\in(0,r)}\gamma(\tau)\quad\mbox{ and }\quad
\limsup_{r\to 0}\gamma(r)=\lim_{r\to 0}\esssup\limits_{\tau\in(0,r)}\gamma(\tau)
\]
if and only if
for all $\alpha>\liminf_{r\to 0}\gamma(r)$ and all  $\beta<\limsup_{r\to 0}\gamma(r)$ the sets
\[
\{r>0:\gamma(r)<\alpha\}\quad\mbox{ and }\quad \{r>0: \gamma(r)>\beta\}
\]
have positive Lebesgue measure. This condition obviously holds
if $\gamma$ is a continuous or a monotone function on $(0,R)$.
However, it may fail for an oscillating semi-continuous function. For
instance, consider the following function:
\[
\gamma(r)=\begin{cases}
1 & \mbox{for }r=\tfrac1n,\ n\geq 1;\\
0 & \mbox{otherwise.}
\end{cases}
\]
\end{remark}

\begin{proof} We give the proof for the case $\underline\beta\geq 1-N$ only; the other cases being similar.

Let us first note that for any $g\in GL_N$ we have
\[
\begin{split}
a_g(x)& =ga(g^{-1}x)g^\top = gg^\top
+ \gamma(|g^{-1}x|)\frac {x\otimes x}{|g^{-1}x|^2},\\
b_g(x)& =gb(g^{-1}x) = \beta(|g^{-1}x|)\frac {x}{|g^{-1}x|^2}.
\end{split}
\]
Thus, with $\gamma$ and $\beta$ standing for $\gamma(|g^{-1}x|)$ and $\beta(|g^{-1}x|)$,
respectively,
\[
\Psi_{g}(x)= \frac{{\rm Tr}\, (gg^\top) + (\gamma+\beta)\frac{|x|^2}{|g^{-1}x|^2} }
{\frac{|g^\top x|^2}{|x|^2} + \gamma\frac{|x|^2}{|g^{-1}x|^2}}
= \frac{\frac{{\rm Tr}(gg^\top)}{\frac{|x|^2}{|g^{-1}x|^2}} +  \gamma+\beta}
{\frac{|g^\top x|^2|g^{-1}x|^2}{|x|^4} +\gamma}.
\]
Let $\lambda_{\rm{min}}$ and $\lambda_{\rm{max}}>0$ be the minimal and the maximal singular values of the matrix $g$, that is, their squares are the correspondent eigenvalues of $gg^\top$ and $g^\top g$. Then we have
\[
\lambda_{\rm{min}}^2\leq \frac{|x|^2}{|g^{-1}x|^2}  \leq \lambda_{\rm{max}}^2,
\]
with equality on one side if $x$ is an eigenvector of $gg^*$ corresponding to $\lambda_{\rm min}^2$, respectively $\lambda_{\rm max}^2$. Moreover, by the Kantorovich inequality (see, e.g. \cite[Theorem 6.27]{fzhang}) and Cauchy-Schwartz inequality, we find
\begin{equation}\label{matr-est}
1\leq \frac{|g^\top x|^2|g^{-1}x|^2}{|x|^4} \leq \sigma(g):=\frac14\Big(\frac{\lambda_{\rm{max}}}{\lambda_{\rm{min}}}
+ \frac{\lambda_{\rm{min}}}{\lambda_{\rm{max}}}\Big)^2.
\end{equation}

 Obviously,~\eqref{matr-est} becomes
an equality for any $x$ if $g=\lambda\mathrm{I}.$
Hence
\[
1+\frac{1}{\sigma(g)+\gamma}
\left(\frac{{\rm Tr}(gg^\top)}{\lambda_{\rm max}^2} - \sigma(g) + \beta\right)\leq \Psi_{g}(x)\leq 1 + \frac{1}{1+\gamma}
\left(\frac{{\rm Tr}(gg^\top)}{\lambda_{\rm min}^2} - 1 + \beta\right),
\]
with equality on one side if $x$ is an eigenvector of $gg^*$ corresponding to $\lambda_{\rm min}$, respectively to $\lambda_{\rm max}$.
It follows that
\[
\begin{split}
\lim_{r\rightarrow 0} \esssup\limits_{x\in B_r\setminus\{0\}}\Psi_{g}(x) & \leq 1 +\frac{1}{1+ \underline\gamma}
\left(\frac{{\rm Tr}(gg^\top)}{\lambda_{\rm{min}}^2} -1 + \overline\beta\right),\\
\lim_{r\rightarrow 0} \essinf\limits_{x\in B_r\setminus\{0\}}\Psi_{g}(x) & \geq 1 +\frac{1}{\sigma(g)+ \overline\gamma}
\left(\frac{{\rm Tr}(gg^\top)}{\lambda_{\rm{max}}^2} - \sigma(g) + \underline\beta\right).
\end{split}
\]
Finally, observe that
\[
\min\limits_{g\in GL_N}\frac{{\rm Tr}\, (gg^\top)}{\lambda^2_{\rm{min}}} = \max\limits_{g\in GL_N}\frac{{\rm Tr}\, (gg^\top)}{\lambda^2_{\rm{max}}} = N \mbox{ and }\min\limits_{g\in GL_N}\sigma(g)=1,
\]
with all extrema attained for matrices $g$ such that $gg^\top=\lambda \rm{I}$ for some $\lambda>0$.
Hence
\[
\overline{\overline{\Psi}} =1 + \frac{N-1+\overline\beta}{1+\underline\gamma}\quad\mbox{and }\quad
\underline{\underline{\Psi}}= 1+ \frac{N-1+\underline\beta}{1+\overline\gamma}.
\]
\end{proof}

\section{Emden-Fowler-type equation}

The proof of Theorem \ref{main_technical} relies essentially on the study of the following final value problem for the ODEs:
\begin{equation}\label{aaw4}
\left\{
\begin{aligned}
&v''+\frac{\phi(r) }{r}v' =\theta(r)|v|^{p-1}v\quad \mbox{ on } (0, R),\\
&v(R)=M \,,\,v'(R)=\lambda\,,
\end{aligned}
\right.
\end{equation}
where $R>0$, $M\geq0$, $\lambda\in\mathbb{R}$, $\phi\in L^\infty(0,R)$ and
$\theta:(0,R)\rightarrow (0,\infty)$ is a measurable function such that $\essinf\theta>0$ and $\int_\varepsilon^R\theta(r)dr<\infty$ for all $\varepsilon>0$ small.

We introduce the following notation:
\[
\Gamma(r):=\mathrm{exp}\Big\{-\int_r^R\phi(\tau)\frac{d\tau}\tau\Big\}
\quad\mbox{ and }\quad t(r):=\int_r^R\frac{d\rho}{\Gamma(\rho)}.
\]
\begin{theorem}\label{emden-fowler-thrm}
(i) Assume that
\begin{equation}\label{superpessimal}
\int_0^{R}\theta(r)\Gamma(r)t^p(r)dr<\infty.
\end{equation}
Then there exists $M>0$ and $\lambda\leq0$ such that the
(locally unique) solution $v$ to~\eqref{aaw4} can be extended
to the interval $(0,R)$ and $v(r)\to\infty$ as $r\to0$.

(ii) Assume that
\begin{equation}\label{superoptimal}
\int_0^R\theta(r)\Gamma^{1-p}(r)r^pdr =\infty.
\end{equation}
Then, for every $M>0$ and $\lambda\leq0$, the
(locally unique) solution $v$ to~\eqref{aaw4} cannot be extended
to the interval $(0,R)$, that is, there exists $R'\in(0,R)$ such that
the solution $v$ can be extended to the interval $(R',R)$ and
$v(r)\to\infty$ as $r\to R'$.
\end{theorem}

\begin{remark}\label{comparison}
(i) The equation in \eq{aaw4} is equivalent to the following:
\begin{equation}\label{aaw5}
\Big(\Gamma v'\Big)' = \theta\Gamma |v|^{p-1}v.
\end{equation}

(ii) If $\lambda\leq0$ then $v$ is a positive decreasing function.

(iii) For every $M>0$ there exists $\lambda>0$ and $R'\in(0,R)$ such that $v(R')>0$ and $v'(R')=0$. Indeed, assume the contrary. Then there exists $M>0$ such that, for every $\lambda>0$, the solution $v_\lambda$ of~\eqref{aaw4} satisfies $v_\lambda'(r)>0$ on the interval $\{r:v_\lambda(r)>0\}$. Since $v_\lambda$ is continuous in $\lambda$, it follows that $v_0$ (the solution to~\eqref{aaw4} with $\lambda=0$) is a non-decreasing function in a neighborhood of $R$. However, this contradicts to (ii).

(iv) For $p\geq 0$, the functions $v$ and $-v'$ increase in $M$ and $\theta$ and decrease in $\lambda$. Indeed, assume that the statement is false. Then there exist $M_0<M_1$, $\lambda_0>\lambda_1$, $\theta_0\leq\theta_1$ the corresponding solutions $v_0$, $v_1$ and $R'\in (0,R)$ such that $v_0(r)<v_1(r)$ and $v_0'(r)>v_1'(r)$ for $r\in(R',R)$ and $v_0(R')=v_1(R')$ or $v_0'(R')=v_1'(R')$.
Then,
\[
v_0(R')-v_1(R') = M_0 - M_1 - \int_{R'}^{R}\big(v_0'(r)-v_1'(r)\big)dr<0.
\]
and, by~\eqref{aaw5},
\[
\begin{split}
v_0'(R') &- v_1'(R') =\\
& = \frac{\lambda_0-\lambda_1}{\Gamma(R')} + \frac1{\Gamma(R')}\int_{R'}^{R}\Gamma(r)\big(\theta_1(r)|v_1(r)|^{p-1}v_1(r)-\theta_0(r)
|v_0(r)|^{p-1}v_0(r)\big)dr>0.
\end{split}
\]
\end{remark}

The proof of Theorem~\ref{emden-fowler-thrm}(i) is divided into several propositions,
partly inspired by~\cite{kig}.

\begin{lemma}\label{exist_vel0} Assume that~\eqref{superpessimal} holds and let $p>1$.
Then there exists $M>0$ such that the (locally unique) solution to~\eqref{aaw4} with $\lambda=0$ is a decreasing function which can be extended to the interval $(0,R)$.
\end{lemma}
\begin{proof}
Due to Remark~\ref{comparison}(ii), we are left to prove that $v$ can be extended to the interval $(0,R)$.
Integrate~\eqref{aaw5} to obtain the following:
\begin{equation}\label{pre-gronwall}
\begin{split}
v(r) = & M - \int_r^{R}v'(\varrho)d\varrho = M + \int_r^{R}\frac1{\Gamma(\varrho)}\int_{\varrho}^{R}\theta(\rho)\Gamma(\rho)v^{p}(\rho)d\rho\, d\varrho
\\
= & M + \int_r^{R}\theta(\rho)\Gamma(\rho)v^{p}(\rho) \int_r^\rho\frac1{\Gamma(\varrho)}d\varrho\,d\rho
\\
\leq & M + \int_r^{R}\frac1{\Gamma(\varrho)}d\varrho \int_r^{R}\theta(\rho)\Gamma(\rho)v^{p}(\rho)d\rho
\\
= & M + t(r)\int_r^{R}\theta(\rho)\Gamma(\rho)v^{p}(\rho)d\rho.
\end{split}
\end{equation}
Then \eqref{pre-gronwall} implies the following bound:
\[
v(r)\leq M + t(r)V(r)
\]
with
\[
V(r):= \int_r^{R}\theta(\rho)\Gamma(\rho)v^{p}(\rho)d\rho.
\]
Since $v$ is a decreasing function, one has either $v< 2M$ on $(0,R)$ (in particular, then
$v$ can be extended to the interval $(0,R)$ as a bounded solution to~\eqref{aaw4}), or
there exists a unique $r_0\in(0,R)$ such that $v(r_0)=2M$. Hence
$v(r)-M\geq \frac12 v(r)$ for $r<r_0$ and
\begin{equation}\label{gronwall_ex}
    v(r)\leq 2t(r)V(r)\quad \mbox{ for } r<r_0.
\end{equation}

It follows that
\[
-V'(r)=\theta(r)\Gamma(r)v^{p}(r)\leq 2^p\theta(r)\Gamma(r)t^p(r)V^p(r)\quad  \mbox{ for } r<r_0.
\]
Let
\[
\Phi(r):= 2^p\int_r^{R}\theta(\rho)\Gamma(\rho)t^p(\rho)d\rho.
\]
Note that~\eqref{superpessimal} implies $\Phi(r)<\Phi(0)<\infty$.

The next estimate holds:
\[
V^{1-p}(r_0) - V^{1-p}(r) \leq (p-1)\Phi(r)\leq (p-1)\Phi(0) \quad \mbox{ for } r<r_0.
\]
So
\[
V(r)\leq \Big(V^{1-p}(r_0) - (p-1)\Phi(0)\Big)^{-\frac1{p-1}} \quad \mbox{ for } r<r_0,
\]
provided
\begin{equation}\label{prov}
V(r_0)< \Bigg((p-1)2^p\int_0^{R}\theta(\rho)\Gamma(\rho)t^p(\rho)d\rho\Bigg)^{-\frac1{p-1}}.
\end{equation}
Since $v\in (M,2M)$ on $(r_0,R_0)$, it follows that
\[
V(r_0)=\int_{r_0}^{R}\theta(\rho)\Gamma(\rho)v^{p}(\rho)d\rho \leq (2M)^p\int_{0}^{R}\theta(\rho)\Gamma(\rho)d\rho.
\]
Thus~\eqref{prov} holds for a sufficiently small $M$ since~\eqref{superpessimal} holds.
Hence we conclude that $V$ uniformly bounded on $(0,R)$. Finally,~\eqref{gronwall_ex} implies
that $v$ can be extended to $(0,R)$ as a solution to~\eqref{aaw4}.
\end{proof}

\begin{lemma}\label{blow-up}
For every $p>1$ and $M\geq0$ there exists $\lambda<0$ such that the (locally unique) solution $v$ to~\eqref{aaw4} cannot be extended to the interval $(0,R)$, that is, there exists $R'\in(0,R)$ such that
$v$ can be extended to the interval $(R',R)$ and $v(r)\to\infty$ as $r\to R'$.
\end{lemma}

\begin{proof}
Assume that $v$ can be extended to the interval $(0,R)$ as a solution to~\eqref{aaw4} for every $\lambda<0$. Note that change of variables $t=t(r)$ is a diffeomorphism $(0,R)\to (0,T)$ with $T=t(0)\in (0,\infty]$, which transforms~\eqref{aaw4} into the following initial value problem:
\begin{equation}\label{aaw6}
\begin{cases}
v^{\prime\prime} = \omega(t)|v|^{p-1}v & \quad \mbox{ on }(0,T),\\
v(0)=M,\quad v^\prime(0)=-\lambda,&
\end{cases}
\end{equation}
with $\omega(t):=\theta\big(r(t)\big)\Gamma^2\big(r(t)\big)$ independent on $\lambda$.

Similarly to Remark~\ref{comparison}(iv), $v$ and $v'$ increase in $M$ and $\omega$, and decrease in $\lambda$. Hence, it suffices to consider the case $\omega\leq 1$ and $M=0$.

Multiply~\eqref{aaw6} by $2v'>0$ and integrate from $0$ to $t$. Since $\omega\leq1$, one has
\neweq{split1}
\begin{aligned}
|v'(t)|^2  &= \lambda^2 + 2\int_0^{t}\omega(\tau)v^p(\tau)v'(\tau)d\tau
\\
& \leq  \lambda^2 + 2\int_0^{t}v^p(\tau)v'(\tau)d\rho
\leq \lambda^2 +\tfrac2{p+1} v^{p+1}(t).
\end{aligned}
\endeq
Note that, for $\lambda\leq0$, the solution $v$ is a convex increasing function. Let $0<S<T$. Then, for all $t\in (S,T)$ we have
\[
v(t)=v(t)-v(0)\geq tv'(0)\geq S|\lambda|.
\]
Therefore, for $|\lambda|>S^{-\frac{p+1}{p-1}}$,
\[
\lambda^2 \leq |S\lambda|^{p+1} \leq v^{p+1}(t) \quad  \mbox{ for all }  t\in(S,T).
\]
Uing this fact in \eq{split1},  there exists $c>0$ such that
\[
|v'(t)|^2 \leq c^2 v^{p+1}(r) \quad \mbox{ for all } t\in(S,T).
\]
Thus, integrating in \eqref{aaw6} it follows that
\[
|\lambda| + \int_0^{t}\omega(\tau)v^p(\tau)d\tau = v'(t) \leq c v^{\frac{p+1}2}(t) \quad \mbox{ for all } t\in(S,T).
\]
Now consider the function $V$ defined as
\[
V(t):= |\lambda| + \int_{S}^t\omega(\tau)v^p(\tau)d\tau \quad \mbox{ for all } t\in(S,T).
\]
The preceding estimates yield
\[
V(t)\leq v'(t)\leq c\Bigg(\frac{V'(t)}{\omega(t)}\Bigg)^{\frac{p+1}{2p}}
\Longleftrightarrow
\omega(t)\leq C V'(t)V^{-\frac{2p}{p+1}}(t) \quad \mbox{ for all } t\in(S,T),
\]
with $C=c^{\frac{2p}{p+1}}>0$. Hence
\[
\int_{S}^T\omega(\tau)d\tau \leq C\tfrac{p+1}{p-1}V^{-\frac{p-1}{p+1}}(S)
= C\tfrac{p+1}{p-1} |\lambda|^{-\frac{p-1}{p+1}}.
\]
Since $|\lambda|>S^{-\frac{p+1}{p-1}}$ it follows that
\[
\int_{S}^T\omega(\tau)d\tau \leq C\tfrac{p+1}{p-1}S.
\]
We can now choose $S>0$ sufficiently small such that the above estimate leads to a contradiction.
\end{proof}

\begin{corollary}\label{exist_gen}
Let $p>1$ and $M>0$. Assume that the (locally unique) solution $v_0$ to~\eqref{aaw4} with $\lambda=0$ can be extended to the interval $(0,R)$. Then there exists $\lambda\leq0$ such that the solution $v$ to~\eqref{aaw4} can be extended to the interval $(0,R)$ and $v(r)\to\infty$ as $r\to0$.
\end{corollary}
\begin{proof}
The assertion holds trivially if $v_0(r)\to\infty$ as $r\to0$. Otherwise, for all $\lambda<0$ let $v_\lambda$ denote a (locally unique) solution to~\eqref{aaw4} with $v_\lambda(R)=M$ and  $v'_\lambda(R)=\lambda$. Let
\[
\Lambda = \Big\{\lambda<0: v_\lambda
\mbox{ continues to }r=0\mbox{ as a bounded solution to \eq{aaw4}}
\Big\}.
\]
Due to Lemma~\ref{blow-up}, $\Lambda$ is a bounded interval. Let $\lambda_0=\inf \Lambda>-\infty$.

First we show that $v_{\lambda_0}$ can be extended to the interval $(0,R)$ as a solution to~\eqref{aaw4}. Indeed, assume the contrary. Then there exists $R'\in (0,R)$ such that $v_{\lambda_0}(r)\to+\infty$ and
$v'_{\lambda_0}(r)\to-\infty$ as $r\to R'$. Consider the final value problem:
\[
\left\{
\begin{aligned}
&w''+\frac{\phi(r)}{r}w' =\theta(r)|w|^{p-1}w \mbox{ on } (0, R')\,,\\
&w(R')=0 \,,\,w'(R')=\lambda'\,,
\end{aligned}
\right.
\]
By Lemma~\ref{blow-up}, there exists $\lambda'<0$ such that $w$ cannot be extended to the interval $(0,R')$, that is, it blows up on the interval $(0,R')$. On the other hand, there exists $\lambda_1\in \Lambda$ such that $v_{\lambda_1}'(R') < \lambda'$ since otherwise $v'_{\lambda_0}(R')\geq \lambda'$. Hence,
by Remark~\ref{comparison}(iv), $v_{\lambda_1}(r)>w(r)$ on $(0,R')$ and so $w$ cannot blow up.
This contradiction proves that $v_{\lambda_0}$ can be extended to interval $(0,R)$ as a solution to~\eqref{aaw4}.

Finally, if $v_{\lambda_0}$ is bounded on $(0,R)$, then the continuous dependence of $v$
in $\lambda$ implies the existence of $\lambda<\lambda_0$ such that $v_\lambda$ can be extended to the interval $(0,R)$ as a bounded solution to~\eqref{aaw4}. Then $\lambda\in \Lambda$ which contradicts the definition of $\lambda_0$. Thus, $v_{\lambda_0}(r)\to\infty$ as $r\to0$.
\end{proof}

\begin{proof}[{\bf Proof of Theorem~\ref{emden-fowler-thrm} completed}] (i)
The assertion follows from Lemma~\ref{exist_vel0} and Corollary~\ref{exist_gen}.
(ii)
Due to Remark \ref{comparison}(iv), it suffices to consider the case $\theta\leq1$.

Assume that, contrary to the assertion, $v$ can be extended to the interval $(0,R)$ as a solution of~\eqref{aaw4}. It follows from~\eqref{rad1} that $\widehat v(x):=v(|x|)$ satisfies the equation
\[
\sum_{i,j=1}^N\widehat a_{ij}(x)\frac{\partial^2 \tilde v}{\partial x_i\partial x_j}
+ \sum_{k=1}^N\widehat b_{k}(x)\frac{\partial \tilde v}{\partial x_k}
=\tilde v^p\quad\mbox{ in }B_R\setminus\{0\}\,,
\]
where
\[
\begin{aligned}
\widehat a_{ij}(x) & =\tfrac1{\theta(|x|)}
\left\{\delta_{ij}+\frac{\phi_+(|x|)}{N-1}\Big[\delta_{ij}-\frac{x_ix_j}{|x|^2}  \Big]  \right\},&& i,j &=1,2,\ldots, N,
\\
\widehat b_k(x) & =-\tfrac1{\theta(|x|)}(\phi_-(|x|) + N - 1)\frac{x_k}{|x|^2},&&k&=1,2,\ldots, N.
\end{aligned}
\]
Note that $\mathbf{\widehat a}:=\{\widehat a_{ij}\}_{i,j=1}^N$ and $\mathbf{\widehat b}:=\{\widehat b\}_{k=1}^N$
satisfy~\eqref{elliptic} and~\eqref{bvec}, respectively, since
$\theta^{-1}\in L^\infty(0,R)$. By Proposition~\ref{KO} it follows that there exists $c>0$ such that
\[
\widehat v(x)\leq c |x|^{-\frac2{p-1}}\quad\mbox{ for all }x\in B_R\setminus\{0\}\,.
\]
So $\widehat v$ is a positive solution to the equation
\begin{equation}\label{fuchsian}
\Big(\sum_{i,j=1}^N\widehat a_{ij}(x)\frac{\partial^2 }{\partial x_i\partial x_j}
+ \sum_{k=1}^N\widehat b_{k}(x)\frac{\partial }{\partial x_k} - Q(x)\Big)w =0 \quad \mbox{ in }
B_R\setminus\{0\},
\end{equation}
with $0<Q(x) := \widehat v^{p-1}(x)<c|x|^{-2}$. Therefore the operator in equation~\eqref{fuchsian}
is of Fuchsian type. Hence by the scaling argument (see~\cite{pinchover94}), $v$ satisfies the Harnack inequality:
\begin{equation}\label{harnack}
\mbox{there exists }  C>0 \mbox{ such that }\frac{v(r/2)}{v(r)}< C  \mbox{ for all }r\in(0,R).
\end{equation}
It follows from~\eqref{aaw5} that
\[
-v'(r) = \frac1{\Gamma(r)}\int_r^{R}\theta(\rho)\Gamma(\rho) v^{p}(\rho)d\rho.
\]
Since $-v'(r)\Gamma(r)$ is a decreasing function, we have
\[
v(r/2)-v(r) = -\int_{r/2}^rv'(\rho)d\rho = -\int_{r/2}^r\frac1{\Gamma(\rho)}v'(\rho)\Gamma(\rho)d\rho
\geq -v'(r)\Gamma(r)\int_{r/2}^r\frac{d\rho}{\Gamma(\rho)}.
\]
Furthermore, for all $\rho\in (r/2,r)$ we have
\[
\frac{\Gamma(r)}{\Gamma(\rho)}=\expn{\int_\rho^r\phi(\tau)\frac{d\tau}\tau}
\geq \expn{-\|\phi\|_\infty\int_{r/2}^r\frac{d\tau}\tau}=2^{-\|\phi\|_\infty}. 
\]
Hence $v(r/2)-v(r)\geq -2^{-\|\phi\|_\infty-1}rv'(r)$ and by~\eqref{harnack},
there exists $c>0$ such that
$
-v'(r)<\frac1{cr}v(r).
$
So we have
\[
v(r) \geq \frac{cr}{\Gamma(r)}\int_r^{R}\theta(\rho)\Gamma(\rho) v^{p}(\rho)d\rho
\quad \mbox{ for all } r\in(0,R).
\]
Let now
\[
V(r):=\int_r^{R}\theta(\rho)\Gamma(\rho) v^{p}(\rho)d\rho \quad  \mbox{ for all }
r\in(0,R).
\]
Then there exists $C>0$ such that
\[
-V'(r)\geq C\theta(r)r^{p}\Gamma^{1-p}(r)V^{p}(r)
\quad  \mbox{ for all }  r\in(0,R).
\]
So
\[
\Big(\frac d{dr}V^{1-p}\Big)(r) \geq  (p-1)C\theta(r)r^{p}\Gamma^{1-p}(r)
\quad  \mbox{ for all } r\in(0,R).
\]
Integrating the above inequality over $(0,\frac12R)$
we obtain
\[
V^{1-p}(\tfrac12R)\geq (p-1)C\int_0^{\frac12R}\theta(r)r^{p}\Gamma^{1-p}(r)dr=+\infty,
\]
due to~\eqref{superoptimal}. The contradiction proves that $v$ cannot be extended to
the interval $(0,R)$ as a solution to~\eqref{aaw4}.
\end{proof}

\section{Proof of the main results}

\subsection{ Proof of Theorem~\ref{main_technical}} (i) It follows
from Theorem~\ref{emden-fowler-thrm} that equation~\eqref{rad2}
has a singular decreasing solution for $p>1$ provided
\eqref{superpessimal} with $\phi(r)=\Env\Psi(r)-1$ and
$\theta(r)=\Env\Theta(r)$ holds. Since
$\Gamma(r)=\frac{R}{r\Eenv(r)},$ we conclude
that~\eqref{superpessimal} is the same as \eqref{exist_criterion}.

(ii) Given $R,M>0$, we show that there exists $\lambda>0$ such that a solution to
  the final value problem for equation \eqref{rad3} with $v(R)=M$ and $v'(R)=\lambda$ does not
  continue to $r=0$, provided~\eqref{non-exist_criterion} holds. Then the assertion will follow from
  Proposition~\ref{nonex-0}.

  First observe that the corresponding equation to~\eqref{rad3} is equivalent to~\eqref{aaw4} with $\phi(r)=\Env\Psi(r)-1$ and $\theta(r)=\env\Theta(r)$ on interval $\{r:\:v'(r)\geq0\}$. Hence it follows from Remark~\ref{comparison}(iii) that there exist $\lambda>0$ and $R'\in(0,R)$ such that $v(r)>0$ for
  $r\in[R',R]$, $v'(r)>0$ for $r\in(R',R)$ and $v'(R')=0$. Then it follows from Remark~\ref{comparison}(ii) that $v$ satisfies~\eqref{aaw4} with $\phi(r)=\env\Psi(r)-1$ and $\theta(r)=\env\Theta(r)$ on interval $(0,R')$. Now Theorem~\ref{emden-fowler-thrm} asserts that $v$ does not continue to $r=0$ provided~\eqref{superoptimal} holds. Finally, observe that~\eqref{superoptimal} is
  the same as~\eqref{non-exist_criterion} since $\Gamma(r)=\frac{R}{r\eenv(r)}.$\qed


\subsection{Proof of Theorem \ref{generalpositive} and Theorem~\ref{general}}
This follows from the next three lemmas. In the first one we evaluate the lower critical exponent
$p_*$ defined in \eqref{crit-exp}.

\begin{lemma}\label{lower-critical}
Assume there exist $c>0$ and $\sigma\geq0$ such that $\Env\Theta(r)\leq cr^{-\sigma}$ for $r\in(0,R)$. Then $(1)_p$ has a solution for all $p<1$, i.e. $p_*=-\infty$.
\end{lemma}
\begin{proof}
We look for a solution $u$ to $(1)_p$ for $p<1$ in the form $u(x)=m|x|^{-\alpha}$ for some $m,\alpha>0$.
By~\eqref{rad2}, it suffices for $m$ and $\alpha$ to satisfy
\begin{equation*}
m\alpha r^{-\alpha-2}\big(2+\alpha-\Env\Psi(r)\big)\geq c m^p r^{-\sigma-\alpha p}\quad
\mbox{ for all } r\in(0,R).
\end{equation*}
This is the case when $m^{p-1}$ is small enough and
\[\alpha+2\geq \alpha p + \sigma\quad \mbox{and}\quad \alpha>\esssup\limits_{(0,R)} \Env\Psi(r) - 2.\]
Since $p<1$, the latter holds for a sufficiently large $\alpha$.
\end{proof}

The lower bound of $p^*$ in~\eqref{ppqqw} follows from the next lemma.

\begin{lemma}\label{exist} Assume $\lim\limits_{\rho\to
0}\essinf\limits_{r<\rho}\Env\Psi(r) >2$ and there exists $0\leq \sigma<2$ such that  
\neweq{k1}
\limsup_{|x|\rightarrow 0} |x|^{\sigma+\varepsilon} K(x)<\infty,
\endeq
for all $\varepsilon>0$.
Then $(1)_p$ has a singular solution for all
$1<p<1+\frac{2-\sigma}{(\overline\Psi-2)_+}$, i.e. $p^*\geq
1+\frac{2-\sigma}{(\overline\Psi-2)_+}$.
\end{lemma}
\begin{proof} We verify \eqref{non-exist_criterion00}. If $g\in GL_N$, then \eq{k1} implies  
$\limsup_{r\rightarrow 0} r^{\sigma+\varepsilon}\Env \Theta_g (r)<\infty$, for all $\varepsilon>0$.

Since $\Eenv_g(r)\leq c r^{-\mathcal{N}(g)}$ for some $c>0$, for all $0<\varepsilon<2-\sigma$ and all $1<p<1+\frac{2-\sigma-\varepsilon}{(\mathcal{N}(g)-2)_+}$
we have
\[
\int_0^R  \Eenv_g^{p-1}(r)\Env\Theta_g(r) r^{2p-1}dr \leq C \int_0^R
r^{2-\sigma - \varepsilon - (p-1)(\mathcal{N}(g)-2)_+}\frac{dr}r<\infty.
\]
By Corollary \ref{coro1} we deduce that $(1)_p$ has a singular solution. We conclude by letting $\varepsilon\rightarrow 0$.
\end{proof}

The upper bound of $p^*$ in~\eqref{ppqqw} follows from the next lemma.

\begin{lemma} Assume there exists $\sigma\geq0$ such that, 
\neweq{k2}
\liminf_{|x|\rightarrow 0} |x|^{\sigma-\varepsilon} K(x)>0,
\endeq
for all $\varepsilon>0$.
Then
\[
p^*\leq 1 + \frac{2-\sigma}{(\underline \Psi-2)_+},\quad\mbox{for }\sigma<2,
\]
and
$p^*=1$ for $\sigma>2$ and for $\sigma=2$, $\underline \Psi>2$.
\end{lemma}
\begin{proof} We verify \eqref{non-exist_criterion}. To this
aim, let $g\in GL_N$ and note that $\eenv_g(r)\geq c r^{-n(g)}$ for some $c>0$.
Thus, using \eq{k2},  for every $\varepsilon>0$, there exist $C_\varepsilon>0$ such that
\[
\int_0^R\eenv_g^{p-1}(r)\env\Theta_g(r)r^{2p-1}dr \geq
C_\varepsilon\int_0^R r^{2p-\sigma+\varepsilon-1-(p-1)n(g)} dr=\infty,
\]
provided $2-\sigma + \varepsilon-(p-1)(n(g) -2 -\varepsilon)<0$. This choice of
an appropriate $\varepsilon>0$ is possible if
\begin{equation}\label{nonex-exponents}
2-\sigma -(p-1)(n(g) - 2)<0.
\end{equation}
For $\sigma<2$ and $n(g)>2$, \eqref{nonex-exponents} is equivalent to
\[
p>1+\frac{2-\sigma}{n(g) - 2}.
\]
For $\sigma=2$ and $n(g)>2$, ~\eqref{nonex-exponents} holds for all $p>1$.
Finally, for $\sigma>2$ and $n(g)\in\mathbb{R}$, there exists
$p\in(1,1+\delta)$ such that~\eqref{nonex-exponents} holds. Hence $p^*=1$ by
~\eqref{crit-exp}.
\end{proof}

\subsection{Proof of Theorem \ref{critic}}

(i) Note that~\eqref{non-exist_criterion} with $p=1+\frac{2-\sigma}{A-2}$ holds in virtue of~\eqref{optimal}.

(ii) We easily verify that \eq{pessimal} follows from
\eqref{non-exist_criterion00} with $p=1+\frac{2-\sigma}{A-2}$.

(iii) Since $h\in L^\infty(0,R)$, it follows from~\eqref{22} that
\[
\int_0^R h^\epsilon(r)\frac{dr}r=\infty
\]
for all $\epsilon\in(0,\varepsilon).$ Hence~\eqref{non-exist_criterion} implies that
$(1)_p$ has no singular solutions for all $p\in(1,1+\varepsilon]$. Then $p^*=1$
by definition~\eqref{crit-exp}.\qed

\subsection{Proof of Theorem~\ref{unstable}}
Let $\alpha\geq 2$ be such that $q=1+\frac{2}{(\alpha-2)_+}$; in particular $q=\infty$ for $\alpha=2$.  Consider the operator ${\mathscr L}$ in the form \eq{operator} where ${\bf b}\equiv 0$ and ${\bf a}$ is defined by \eq{matgs} with
\[
\gamma(r)=-1+\frac{N-1}{\phi(\ln(1/r))}\,,\quad r\in(0,1),
\]
and $\phi:(0,\infty)\to \mathbb{R}$ is a bounded measurable function
such that $\phi$ is 1-periodic and
\begin{equation}\label{periodic}
\int_n^{n+1}\phi(t)dt=\alpha-1
\quad\mbox{ for all integers } n\geq 0. 
\end{equation}
By Proposition~\ref{computations},
\[
\Psi(x)
= 1 + \frac{N-1}{1+\g(|x|)}
=1+\phi\Big(\ln\frac{1}{|x|}\Big)
\quad\mbox{ for all }x\in B_1\setminus\{0\}.
\]
Also
\[
\int_r^{1}\frac{\Psi(\tau)}{\tau} d\tau=\int_r^{1}\frac{1+\phi(\ln(1/\tau))}{\tau} d\tau=\int_0^{\ln(1/r)}\phi(t)dt + \ln\frac1r.
\]
Note that, by \eqref{periodic},
\[
\int_0^{\ln(1/r)}\phi(t)dt - (\alpha-1)\ln\frac{1}{r} =
\int_{[\ln(1/r)]}^{\ln(1/r)}\big(\phi(t) - \alpha +1\big)dt.
\]
Hence
\[
\alpha\ln\frac{1}{r} - \|\phi-\alpha+1\|_{L^\infty(0,1)}\leq \int_r^1 \frac{\Psi(\tau)}{\tau} d\tau\leq
\alpha\ln\frac{1}{r} + \|\phi-\alpha+1\|_{L^\infty(0,1)}.
\]
Thus, if $\mathcal{M}(r)$ and $m(r)$ are defined by \eq{MandN} we have $\mathcal{M}(r)\sim m(r)\sim r^{-\alpha}$.  Since $K\in L^\infty(B_R)$ and $\essinf K>0$, we also have $\Env \Theta(r)\sim \env\Theta(r)\sim 1$. Now from \eq{non-exist_criterion} and \eq{non-exist_criterion00} 
it follows that $(1)_p$ has a singular solution if
and only if $p<1+\frac2{(\alpha-2)_+}$, that is, $p^*=q$.\qed

\section{Examples}\label{examples}

This part presents some applications to our main results in Section 1. For the sake of clarity we shall assume in the following that $K(x)=|x|^{-\sigma}$, $\sigma\geq 0$.
\subsection{Stabilizing coefficients}

\begin{example}\label{exa1}
Consider the inequality
\neweq{inq1}
\sum_{i=1}^N (1+x_i^2)^k\,\frac{\partial ^2 u}{\partial x^2_i}\geq |x|^{-\sigma} u^p\quad\mbox{ in }B_R\setminus\{0\},
\endeq
\end{example}
where $k\in \R$ and $\sigma\geq 0$.
\begin{proposition} Assume $N\geq 3$. Then, inequality \eq{inq1} has singular solutions if and only if $p<1+(2-\sigma)_+/(N-2)$.
\end{proposition}
\begin{proof}
Let ${\bf a}_{ij}(x)=(1+x_i^2)^k\delta_{ij}$, $i,j=1,2,\dots,N$. Then
\[
\Psi_g(x)=\frac{|x|^2}{|(g{\bf a}^{1/2})x|^2} {\rm Tr}\big[(g{\bf a}^{1/2})(g{\bf a}^{1/2})^\top \big]
\]
and with the same arguments as in the proof of Lemma \ref{computations} we find $\underline\Psi=\overline \Psi=N$. Thus, by Theorem \ref{general} it follows that $p^*=1+(2-\sigma)_+/(N-2)$.
We next study the existence of a singular solution in the critical case $p=1+(2-\sigma)_+/(N-2)$.  To this aim, let us remark that
\[
\Psi(x)=\frac{|x|^2\sum\limits_{i=1}^N  (1+x_i^2)^k }{\sum\limits_{i=1}^N x_i^2(1+x_i^2)^k} \quad\mbox{ for all }x\in B_R\setminus\{0\}.
\]
If $k\leq 0$ we use Chebyshev's inequality (see, e.g., \cite[Theorem 43, page 43]{hardy}) to deduce $\Psi(x)\geq N$. 
If $k> 0$ then for all $x\in B_R\setminus\{0\}$ we have
\[
\Psi(x)\geq \frac{N|x|^2}{\sum\limits_{i=1}^N x_i^2(1+x_i^2)^k}\geq \frac{N}{(1+|x|^2)^k}\geq N+N[1-(1+|x|^2)^k].
\]
Also there exists $C=C(N,k,R)>0$ such that
\[
N[1-(1+|x|^2)^k]\geq -C|x|^2\quad\mbox{ for all }x\in B_R\setminus\{0\}.
\]
We obtained that in both cases $k\leq 0$ and $k>0$ there exists a positive constant $C>0$ such that
\[
\Psi(x)\geq N-C|x|^2\quad\mbox{ for all }x\in B_R\setminus\{0\}.
\]
Then 
\[
\displaystyle m(r)=\exp\left\{\displaystyle \int_r^R\frac{\Psi(s)}{s}ds\right\}\ge C(R,N,k) r^{-N}\quad\mbox{ for all }r\in (0,R).
\] 
By Theorem \ref{critic}(i) (take $h\equiv C(N,R,k)$) inequality \eq{inq1} has no solutions in the critical case $p=1+(2-\sigma)_+/(N-2)$.
\end{proof}

\subsection{Gilbarg-Serrin matrices}

We focus next on matrices ${\bf a}$ defined by~\eqref{matgs} in Lemma~\ref{computations}. They are related to Gilbarg-Serrin matrices suggested in \cite{GS,KLS1,MeSer} and provide a rich source of interesting examples as we illustrate in the following.

\begin{example}\label{serrin1}
Consider the inequality
\neweq{ser1}
\Delta u+\gamma(|x|)\sum_{i,j=1}^N \frac{x_ix_j}{|x|^2}\frac{\partial ^2 u}{\partial x_i\partial x_j}\geq |x|^{-\sigma}u^p\quad\mbox{ in }B_R\setminus\{0\}\subset \R^N,
\endeq
where $N\geq 3$ and $\sigma\geq 0$. Assume that $\gamma:(0,R)\rightarrow \R$ is bounded and continuous and satisfies $\limsup_{r\rightarrow 0}\gamma>-1$. This last condition on $\gamma$ ensures the uniform ellipticity of the matrix ${\bf a}$ as required in \eq{elliptic}. 
\end{example}

From Theorem \ref{general} we obtain:
\begin{proposition}\label{exs1} Assume $\limsup_{r\rightarrow 0}\gamma(r)<N-2$. Then, there exists $p^*\geq 1$ such that \eq{ser1} has singular solutions for $p<p^*$ and no singular solutions exist if $p>p^*$. Furthermore, $p^*=1$ if $\sigma\geq 2$ and 
\neweq{pstar}
\frac{(N-\sigma)+(1-\sigma)\liminf\limits_{r\rightarrow 0}\gamma(r)}{N-2-\liminf\limits_{r\rightarrow 0}\gamma(r)}\leq p^*\leq
\frac{(N-\sigma)+(1-\sigma)\limsup\limits_{r\rightarrow 0}\gamma(r)}{N-2-\limsup\limits_{r\rightarrow 0}\gamma(r)}\;\mbox{ if }0\leq \sigma<2.
\endeq
\end{proposition}

In the critical case we have:
\begin{proposition}\label{exs1b} Assume $\lim_{r\rightarrow 0}\gamma(r)=0$ and $0\leq \sigma<2$. Then:
\begin{itemize}
\item[(i)] $p^*=(N-\sigma)/(N-2)$ and \eq{ser1} has singular solutions for $p=p^*$ if and only if
\neweq{ser11}
\int_0^R{\rm exp}\left\{-\frac{(2-\sigma)(N-1)}{N-2}\int_r^R \frac{\gamma(t)}{t(1+\gamma(t))}dt\right\}\frac{dr}{r}<\infty.
\endeq

\item[(ii)] If $\gamma$ is differentiable on $(0,R)$ and there exists $c>0$ such that
\neweq{ser12}
\gamma(r)\leq cr\gamma'(r) \quad\mbox{ for all } 0<r<R,
\endeq
then \eq{ser1} has no singular solutions for the critical exponent $p=(N-\sigma)/(N-2)$.
\end{itemize}
\end{proposition}
\begin{proof}  (i) Since $\lim_{r\rightarrow 0}\gamma(r)=0$, from \eq{pstar} we have $p^*=(N-\sigma)/(N-2)$. Also
\[
\Psi(x)=N-\frac{(N-1)\gamma(|x|)}{1+\gamma(|x|)},\quad x\in B_R\setminus\{0\}.
\]
Condition \eq{ser11} is now a reformulation of \eq{optimal} and \eq{pessimal} with 
\[
h(r)=H(r)=R^N\exp\left\{  -(N-1)\int_r^R \frac{\gamma(t)}{1+\gamma(t)}dt\right\}.
\]

(ii) If $\gamma$ satisfies \eq{ser12} then the integral in \eq{ser11} is divergent since
\[
\int_r^R \frac{\gamma(t)}{t(1+\gamma(t))}dt\leq c\int_r^R \frac{\gamma'(t)}{1+\gamma(t)}dt=C(R)-c\ln(1+\gamma(r)),
\]
for all $0<r<R$. Thus,
\[
\int_0^R{\rm exp}\left\{-\frac{(2-\sigma)(N-1)}{N-2}\int_r^R \frac{\gamma(t)}{t(1+\gamma(t))}dt\right\}\frac{dr}{r}\geq C\int_0^R\frac{(1+\gamma(r))}{r}^{\frac{c(2-\sigma)(N-1)}{N-2}}dr=\infty.
\]
This concludes our proof.
\end{proof}

Let us remark that there are large classes of differentiable functions $\gamma$ satisfying \eq{ser12}. In particular for $\gamma(r)=r^\alpha$, $\alpha\geq 0$, inequality \eq{ser1} has no singular solutions in the critical case $p=(N-\sigma)/(N-2)$, $0\leq \sigma<2$.

We next consider a function $\gamma$ that fails to fulfill \eq{ser12}.

\begin{proposition}\label{exs2} Assume $0\leq \sigma<2$ and let $\gamma(r)=\ln^{-m}\frac{1}{r}$, $m>0$.
\begin{itemize}
\item[(i)] Inequality \eq{ser1} has singular solutions for all $p<(N-\sigma)/(N-2)$ and no singular solutions exist if $p>(N-\sigma)/(N-2)$.
\item[(ii)] If $p=(N-\sigma)/(N-2)$ then \eq{ser1} has singular solutions if and only if either $0<m< 1$ or $m=1$ and $0\leq \sigma<N/(N-1)$.
\end{itemize}
\end{proposition}
\begin{proof}
(i) follows from the first part in Proposition \ref{exs1b}(i).

(ii) Without losing any generality we may assume $R<1/e$. We evaluate the integral in \eq{ser11}. 
If $m>1$ we have
\[
\int_r^R\frac{\gamma(t)}{t(1+\gamma(t))}dt \leq \int_r^R \frac{dt}{t\ln^m\frac{1}{t}}=C-\frac{1}{m-1}\ln^{1-m}\frac{1}{r}.
\]
Hence
\[
\int_0^R{\rm exp}\left\{-\frac{(2-\sigma)(N-1)}{N-2}\int_r^R \frac{\gamma(t)}{t(1+\gamma(t))}dt\right\}\frac{dr}{r}
\geq
C\int_0^R e^{c\ln^{1-m}\frac{1}{r}}\frac{dr}{r}=\infty,
\]
where $c=\frac{(2-\sigma)(N-1)}{(N-2)(m-1)}>0$. 

If $0<m< 1$ we have 
\[
\int_r^R\frac{\gamma(t)}{t(1+\gamma(t))}dt = \int^{\ln(1/r)}_{\ln (1/R)} \frac{ds}{1+s^m}\geq 
\int^{\ln(1/r)}_{\ln (1/R)} \frac{ds}{(1+s)^m}\geq \Big(1+\ln\frac{1}{r}\Big)^{1-m}-C(R).
\]
Hence
\[
\begin{split}
\int_0^R{\rm exp}\left\{-\frac{(2-\sigma)(N-1)}{N-2}\int_r^R \frac{\gamma(t)}{t(1+\gamma(t))}dt\right\}\frac{dr}{r}
& \leq
C\int_0^R e^{-c\big(1+\ln\frac{1}{r}\big)^{1-m}}\frac{dr}{r}\\
&= C\int_{1+\ln(1/R)}^\infty e^{-cs^{1-m}  }ds<\infty,
\end{split}
\]
where $c=\frac{(2-\sigma)(N-1)}{N-2}>0$. 
Finally, if $m=1$ we have
\[
\int_r^R\frac{\gamma(t)}{t(1+\gamma(t))}dt = \int_r^R \frac{dt}{t(1+\ln\frac{1}{t})}=\ln\Big(1+\ln\frac{1}{r}\Big)-C(R).
\]
Thus,
\[
\begin{split}
\int_0^R{\rm exp}\left\{-\frac{(2-\sigma)(N-1)}{N-2}\int_r^R \frac{\gamma(t)}{t(1+\gamma(t))}dt\right\}\frac{dr}{r}
& =
C\int_0^R\frac{\Big(1+\ln\frac{1}{r}\Big)}{r}^{-\frac{(2-\sigma)(N-1)}{N-2}}\!\!\!\!\!\!\!dr\\
& = C\int_{1+\ln(1/R)}^\infty s^{-\frac{(2-\sigma)(N-1)}{N-2}}ds.
\end{split}
\]
The integral in \eq{ser11} is finite if and only if $\frac{(2-\sigma)(N-1)}{N-2}>1$, that is $0\leq \sigma <N/(N-1)$. 
\end{proof}

The following result proves the sharpness of condition \eq{optimal} in Theorem \ref{critic}(i).

\begin{proposition}\label{exs3}
Assume $0\leq \sigma<2$ and let
\[
\gamma(r)=\frac{N-A-\varkappa\ln^{-1}\frac{1}{r}}{A-1-\varkappa\ln^{-1}\frac{1}{r}},
\]
where $A>2$ and $\varkappa>0$. Then \eq{ser1} has singular solutions in $B_{1/e}\setminus\{0\}$ for the critical exponent $p=(A-\sigma)/(A-2)$ if and only if $\varkappa>(A-2)/(2-\sigma)$, that is, if and only if \eq{optimal} holds.
\end{proposition}
\begin{proof}
Note first that
\[
\Psi(x)=\frac{N+\g(|x|)}{1+\g(|x|)}=A-\frac{\varkappa}{\ln \frac{1}{|x|}},
\]
and from \eq{doipsi} we have $\underline\Psi=\overline \Psi=A$. Further,  it is easy to  check that 
$m(r)=C(R)r^{-A}|\ln r|^{\varkappa}$ so that 
condition \eq{optimal} holds for $0<\varkappa\leq \frac{A-2}{2-\sigma}$ (see also \eq{parth}). By Theorem \ref{critic}(i), inequality \eq{ser1} has no singular solutions for the critical exponent $p=p^*=(A-\sigma)/(A-2)$. On the other hand, for all $\varkappa>\frac{A-2}{2-\sigma}$ the function
\[
u(x)=c|x|^{2-A}\ln^{\frac{2-A}{2-\sigma}} \Big(\frac{1}{|x|}\Big),
\]
is a solution of \eq{ser1} with $p=(A-\sigma)/(A-2)$ in $B_R\setminus\{0\}$ for suitable small constant $c>0$.
This proves the optimality of \eq{optimal} in Theorem \ref{critic}(i).
\end{proof}

\begin{example}\label{exb}
Consider the inequality
\neweq{inqb}
\Delta u-\frac{1}{2}\sum_{i,j=1}^N \frac{x_ix_j}{|x|^2}\frac{\partial ^2 u}{\partial x_i\partial x_j}+\beta(|x|)\sum_{i=1}^N\frac{x_i}{|x|^2}\frac{\partial u}{\partial x_i}\geq |x|^{-\sigma}u^p\quad\mbox{ in }B_R\setminus\{0\},\, \sigma\geq 0,
\endeq
where $\beta:(0,\infty)\rightarrow \R$ is continuous and $\lim_{r\rightarrow 0} \beta(r)=0$.

\begin{proposition}\label{peb} Assume $N\geq 2$. 
\begin{itemize}
\item[(i)]The inequality \eq{inqb} has singular solutions for all $p<1+(2-\sigma)_+/(2N-3)$ and no singular solutions exist if $p>1+(2-\sigma)_+/(2N-3)$.
\item[(ii)] Assume $0\leq \sigma<2$ and $p=1+(2-\sigma)/(2N-3)$. Then \eq{inqb} has singular solutions if and only if
\neweq{inqbl}
\int_0^R{\rm exp}\left\{\frac{2(2-\sigma)}{2N-3}\int_r^R \beta(t)\frac{dt}{t}\right\}\frac{dr}{r}<\infty.
\endeq
\end{itemize}
\end{proposition}
\begin{proof}
With similar computations to those in Lemma \ref{computations} we find $\Psi(x)=2N-1+2\beta(|x|)$ and 
$\underline \Psi=\overline \Psi=2N-1$.  The conclusion follows now from Theorems \ref{general} and \ref{critic}.
\end{proof}

\end{example}

\section{Open Problems}

In this section we state some open problems that stem from our study of $(1)_p$.

\noindent{\bf Problem 1.} {\it Can similar results be obtained for more general elliptic operators?}

In other words, asume that the symmetric matrix ${\bf a}$ is only strictly elliptic, that is, \eq{elliptic} is replaces by 
\[
\nu(x)|\xi|^2 \le\sum_{i,j=1}^{N}a_{ij}(x)\xi_i\xi_j\le c\nu(x) |\xi|^2
\mbox{ for all }\xi\in\R^N,
\]
and the vector field ${\bf b}=(b_i(x))_{i=1}^N\in L^\infty_{\mathrm{loc}}(B_R)$ fulfills
\[
|b_i(x)|\le \frac{c\nu(x)}{|x|}\quad\mbox{  for almost all }x\in
B_R\setminus\{0\}\,, 1\le i\le N.
\]
Here $c>1$ is a constant and $\nu>0$ satisfies  $\nu\in L^\infty_{loc}(B_R\setminus\{0\})$.  If $\nu\in L^\infty(B_R)$ then we can use directly our arguments for the study of $(1)_p$. The problem remains open for $\nu\in  L^\infty_{loc}(B_R\setminus\{0\})$.

\medskip

\noindent{\bf Problem 2.}  {\it Is it true that $p^*=-\infty$ for all potentials $K\in L^\infty_{loc}(B_R\setminus\{0\})$ satisfying $\essinf K>0$?} 

This is indeed the case if the behavior of $K(x)$ is restricted by \eq{essentially} to a power-like potential. One may  wish to investigate the value of the lower critical exponent $p^*$ if $K(x)$ decays faster near the origin.

\appendix

\section{}
For $f\in L^\infty(0,R)$ define the following family of averaging operators:
\begin{equation}\label{Avgmu}
\begin{split}
\mathop{Avg_s}f(r) & :=|s|r^{-s}\int_r^{R}f(\tau)\tau^{s-1}d\tau, \quad s<0;
\\
\mathop{Avg_s}f(r) & :=sr^{-s}\int_0^rf(\tau)\tau^{s-1}d\tau, \quad s>0.
\end{split}
\end{equation}

\begin{proposition}
Let $f\in L^\infty(0,R)$, $s>0$ and $\sigma<0$. Then there exist bounded continuous functions
$g$ and $h$ such that, for $r\in(0,R)$,
\[
\int_r^{R}f(\rho)\tfrac{d\rho}\rho = g(r) + \int_r^{R}\mathop{Avg_s}f(\rho)\tfrac{d\rho}\rho
= h(r) + \int_r^{R}\mathop{Avg_\sigma}f(\rho)\tfrac{d\rho}\rho.
\]
\end{proposition}
\begin{proof}
Indeed, by integration by parts,
\[
\int_r^{R}f(\rho)\tfrac{d\rho}\rho =
-\int_r^{R} \rho^{-\sigma} \frac{d}{d\rho} \left(\int_\rho^{R}f(\tau)\tau^{\sigma-1} d\tau\right)
= \tfrac1{|\sigma|}\mathop{Avg_\sigma}f(r) + \int_r^{R}\mathop{Avg_\sigma}f(\rho)\tfrac{d\rho}\rho.
\]
Similarly for $s$.
\end{proof}

\noindent{\bf Acknowledgement.} The first named author acknowledges the financial support from the Royal Irish Academy and from the Romanian Ministry of Education, Research, Youth and Sport (CNCSIS PCCE-55/2008).

\end{document}